\documentclass{amsart}
\usepackage[utf8]{inputenc}
\usepackage{amsmath}
\usepackage{hyperref}
\usepackage{dynkin-diagrams}
\usepackage[shortlabels]{enumitem}
\usepackage{color}

\usepackage{multirow}

\newtheorem{theorem}{Theorem}[section]
\newtheorem{lemma}[theorem]{Lemma}

\theoremstyle{definition}

\theoremstyle{remark}
\newtheorem{remark}[theorem]{Remark}

\numberwithin{equation}{section}

\begin{document}

\title{Rational Angle Sets and Tight T-Designs}
\newcommand{\chapter}[1]{}

\author{Benjamin Nasmith}

\date{\today}



\begin{abstract}
    Given a finite subset of a sphere or projective space, known as a design, we can compute the strength and angle set of that design. When the strength and angle set meet certain bounds, the design is called tight. Hoggar sought to prove that, aside from certain known cases, the angle sets of tight projective designs must be rational. Lyubich found a counter-example and provided a repair for Hoggar's proof but excluded the exceptional octonion projective cases. This note extends Lyubich's repair of Hoggar's proof to the remaining projective cases and extends the proof to all spherical cases. It does so by using Jordan algebra primitive idempotents to treat all of the cases simultaneously. We thereby confirm that tight spherical and projective designs have rational angle sets except in specific cases.
\end{abstract}

\maketitle

Combinatorial $t$-designs were generalized to spherical $t$-designs in \cite{delsarte_spherical_1977} and to projective spaces in \cite{neumaier_combinatorial_1981} (see also \cite{seidel_designs_1990}). 
Given a finite subset $X$ of a sphere or projective space we can evaluate both the angle set $A(X)$ and strength $t$ of that subset. 
For a given strength $t$ there exists an absolute lower bound on the cardinality $|X|$ such that $X$ is a $t$-design. Likewise, for a given cardinality $s$, there is an absolute upper bound on the cardinality $|X|$ such that $|A(X)| = s$. Furthermore, $t$ is bounded by $s$ according to the inequality $t \le 2 s - \varepsilon$ where $\varepsilon = |A \cap \{0\}|$. These three bounds are satisfied simultaneously if any one of them is met. When a set $X$ meets these absolute bounds, $X$ is called a \textit{tight t-design}. 

The full classification of tight $t$-designs is incomplete, but various theorems place upper bounds on the value of $t$ for different geometries (e.g., \cite{hoggar_tight_1989}, \cite{bannai_tight_1989}). 
In the case of projective geometries, many of these theorems constraining $t$ depend on a result given in \cite{hoggar_tight_1984} that, except for the real projective line, the angle set $A(X)$ must be rational.
However, a counter-example exists in the case of the complex projective line: a subset corresponding to the vertices of an icosahedron.
This counter-example is examined in \cite{lyubich_tight_2009}, which attempts to repair the defective proof in \cite{hoggar_tight_1984}. 
Unfortunately, the repair in \cite{lyubich_tight_2009} is restricted to the real, complex, and quaternion projective cases. 
It neglects the octonion projective case and the spherical cases. 
The aim of this note is to complete the repair in \cite{lyubich_tight_2009} by including the remaining octonion and spherical cases. This also generalizes the attempted proof in \cite{hoggar_tight_1984} to the full family of spherical cases. 
In order to treat all possible cases at once, we will work with the primitive idempotents of simple Euclidean Jordan algebras. This allows us to treat the spherical, projective, and octonion cases in a unified way.

\section{Jordan Algebras and T-Designs}

This section reviews simple Euclidean Jordan algebras and the concepts required to identify and describe tight $t$-designs. 
In addition to the real numbers $\mathbb{R}$, the classification of simple Euclidean Jordan algebras consists of four infinite families and one exception. The first infinite family has rank $\rho = 2$ and degree $d \ge 1$. 
The second, third, and fourth family respectively have degree $d = 1,2,4$ and rank $\rho \ge 3$. 
The exceptional Euclidean Jordan algebra has rank $\rho = 3$ and degree $d = 8$. 
Each Euclidean Jordan algebra has a well defined trace that we can use to define a Euclidean inner product,
\begin{align*}
    \langle x, y \rangle = \mathrm{Tr}(x\circ y).   
\end{align*}
Here $\circ$ denotes the Jordan product. 
Let $V$ be a simple Euclidean Jordan algebra of rank $\rho$ and degree $d$. 
We denote by $\mathcal{J}(V)$ the manifold of primitive idempotents of $V$.
The rank $\rho = 2$ family has manifolds of primitive idempotents isometric to spheres. The degrees $d = 1,2,4$ families have manifolds of primitive idempotents respectively isometric to real, complex, and quaternionic projective spaces. Finally, the rank $\rho = 3$, degree $d = 8$ exceptional case has a manifold of primitive idempotents isometric to the octonion projective plane. 
This means we can use simple Euclidean Jordan algebras and their manifolds of primitive idempotents to model the following geometries for $d \ge 1$ and $\rho \ge 3$:
\begin{align*}
    \Omega_{d+1}, 
    \quad \mathbb{RP}^{\rho - 1}, 
    \quad \mathbb{CP}^{\rho-1}, 
    \quad \mathbb{HP}^{\rho - 1}, 
    \quad \mathbb{OP}^2.
\end{align*}
More details about simple Euclidean Jordan algebras can be found in \cite{faraut_analysis_1994}.

Let $X$ be a finite subset of $\mathcal{J}(V)$, the manifold of primitive idempotents of simple Euclidean Jordan algebra $V$. Then $X$ is an \textit{$A$-code} where $A$ is defined as,
\begin{align*}
    A(X) = \left\{ \langle x,y\rangle \mid x\ne y \in X \subset \mathcal{J}(V) \right\}, 
\end{align*}
and $X$ is also a \textit{$t$-design} where $t$ is the largest integer such that $X$ satisfies,
\begin{align*}
    \sum_{x \in X} \sum_{y\in X} Q_k^{0}(\langle x,y\rangle) = 0, \quad k = 1,2, \ldots, t.
\end{align*}
Here we use \textit{renormalized Jacobi functions} $Q_k^\varepsilon(x)$ with $\varepsilon = 0,1$ in terms of the rank $\rho$ and degree $d$ of $V$ as follows (recall that $\Omega_{d+1}$ has rank $2$ and degree $d$):
\begin{align*}
    Q_k^\varepsilon(x) = \left(\frac{\frac{1}{2}\rho d  + 2k + \varepsilon - 1}{\frac{1}{2}\rho d  + k + \varepsilon - 1}\right)\frac{(\frac{1}{2}\rho d )_{k +\varepsilon} }{(\frac{1}{2}d)_{k+\varepsilon} }P_k^{(\frac{1}{2}d (\rho - 1) - 1, \frac{1}{2}d - 1 + \varepsilon)}(2x-1).
\end{align*}
Here and below we use the \textit{Pochhammer symbol} for non-negative integer $n$, which can also be defined in terms of the usual $\Gamma$ function:
\begin{align*}
    (x)_n = x(x+1)\cdots (x+n-1) = \frac{\Gamma(x + n)}{\Gamma(x)}.
\end{align*}
Some of the polynomials that we will use are given below, with $N = \frac{1}{2}\rho d$ and $m = \frac{1}{2}d$. The $N,m$ notation is more common in the literature about projective designs.
\begin{align*}
    Q_0^0(x) &= 1, \\
    Q_0^1(x) &= \frac{N}{m}, \\ 
    Q_1^0(x) &= \left(N +1\right) \left(\frac{N}{m} x - 1\right), \\
    Q_2^0(x) &= \left(\frac{N(N+3)}{2m(m+1)}\right) 
    \left((N(N+3) +2) x^2 -2(N+1)(m+1)x + m(m+1)\right).
\end{align*}
Next we construct the \textit{annihilator polynomial} of $X$ \cite[242]{hoggar_t-designs_1982}:
\begin{align*}
    \mathrm{ann}(x) = \frac{|X|}{\prod_{\alpha\in A} (1-\alpha)} \prod_{\alpha \in A}(x -\alpha).
\end{align*}
By construction, we have $\mathrm{ann}(1) = |X|$ and $\mathrm{ann}(\alpha) = 0$ for each angle $\alpha \in A$. 
The polynomial $\mathrm{ann}(x)$ has degree $|A| = s$, and can be written as a linear combination of our renormalized Jacobi functions (which depend on the rank and degree of the Jordan algebra containing $X$): 
\begin{align*}
    \mathrm{ann}(x) = \sum_{i = 0}^s a_i Q_i^0(x).
\end{align*}
The coefficients $a_0, a_1, \ldots, a_s$ are known as the \textit{indicator coefficients} of $X$. 
To summarize, given a finite subset $X \subset \mathcal{J}(V)$ we can determine the values of $A$ and $t$ needed to describe $X$ as an $A$-code and $t$-design. We can also compute the annihilator polynomial $\mathrm{ann}(x)$ and indicator coefficients $a_0, a_1, \ldots, a_s$. 


A \textit{tight} $(2s - \varepsilon)$-\textit{design} is a finite subset $X \subset \mathcal{J}(V)$ where the annihilator polynomial, as defined above, obtains the following value:
\begin{align*}
    \mathrm{ann}(x) = x^\varepsilon R_{s-\varepsilon}^\varepsilon(x), 
    \quad
    R_{s-\varepsilon}^\varepsilon(x) = \sum_{i = 0}^{s - \varepsilon} Q_i^\varepsilon(x).
\end{align*}
As the inner products of primitive idempotents in a Euclidean Jordan algebra, the elements of $\alpha \in A(X)$ are all real-valued in the range $0 \le \alpha < 1$.  
We are interested in whether a tight $t$-design will have only \textit{rational} elements in angle set $A(X)$.
This note proves the following theorem, which generalizes the theorems of \cite{hoggar_tight_1984} and \cite{lyubich_tight_2009}:
\begin{theorem}
    \label{maintheorem}
    Let $V$ be a simple Euclidean Jordan algebra of rank $\rho$ and degree $d$ with manifold of primitive idempotents $\mathcal{J}(V)$. 
    Let $X$ be a finite subset of $\mathcal{J}(V)$ forming a tight $(2 s -\varepsilon)$-design, namely with $\mathrm{ann}(x) = x^\varepsilon R_{s-\varepsilon}^\varepsilon(x)$. 
    Then the roots of $\mathrm{ann}(x)$, which form the angle set $A(x)$, are rational with exceptions when $(\rho,d) = (2,1)$ with $t \ne 1,2,3,5$ and when $(\rho, d) = (2,2)$ with $t = 5$.
\end{theorem}

\section{Bose-Mesner Algebras and the Idempotent Basis}

We now review the faulty proof given in \cite{hoggar_tight_1984}, which was intended for the degree $d = 1,2,4,8$ cases only (i.e., the projective cases). The notation here is not necessarily the same as that in \cite{hoggar_tight_1984} or \cite{lyubich_tight_2009}. 
The problem with the faulty proof in \cite{hoggar_tight_1984} is that the matrices $E_i$ identified in that paper are not in fact the idempotent basis for the Bose-Mesner algebra that they are assumed to be. The burden of \cite{lyubich_tight_2009} is to replace $E_i$ with the correct idempotents $L_i$ and complete the remainder of the proof, in the case of degrees $d = 1,2,4$. 
We do the same here for any rank and degree.

First, a tight $(2s-\varepsilon)$-design has the property that $t \ge 2s - 2$, which ensures that $X$ defines an \textit{association scheme}.
We can describe an association scheme in terms of the Gram matrix of the elements of $X$ with respect to the Jordan inner product $\langle x,y\rangle = \mathrm{Tr}(x\circ y)$ given above. 
That is, the elements of $G$, a $|X| \times |X|$ matrix, are given by:
\begin{align*}
    (G)_{x,y} = \langle x, y\rangle. 
\end{align*}
We can write this Gram matrix as a linear combination of adjacency matrices as follows: 
\begin{align*}
    G = I + \sum_{\alpha \in A(X)} \alpha D_\alpha.
\end{align*}
Here $D_\alpha$ is the adjacency matrix of the graph on $X$ where an edge exists between any $x,y$ in $X$ with $\langle x,y\rangle = \alpha$.
In this notation, we can write $I = D_1$ since $\langle x , x\rangle = 1$ for all $x$ in $X$. 
Specifically, for $t \ge 2s - 2$ (which is satisfied for tight $t$-designs), the $D_\alpha$ for $\alpha$ in $A(X)$ define the $s = |A(X)|$ classes of an association scheme. 

The matrices $D_\alpha$ and $I$ form the basis for a commutative matrix algebra of dimension $s + 1$ known as the \textit{Bose-Mesner algebra} of $X$ (or rather of the association scheme defined on $X$ via its Gram matrix). 
The Bose-Mesner algebra consists of all $\mathbb{C}$-linear combinations of the commuting basis given by the adjacency matrices and the identity matrix,
\begin{align*}
    \{D_\alpha \mid \alpha \in A(X)\} \cup \{I\}
\end{align*}
The simultaneous eigenvectors of these commuting diagonal matrices can be used to construct a unique orthogonal idempotent basis for the Bose-Mesner algebra \cite[pp. 201-204]{cameron_designs_1991}:
\begin{align*}
 \{L_i \mid i = 0, 1, \ldots, s\}, \quad L_i L_j = \delta_{i,j} L_i.
\end{align*}
We denote by $q_i(\alpha)/|X|$ the coefficients of the $L_i$ elements in the $D_\alpha$ basis, such that:
\begin{align*}
    |X| L_i = q_i(1) I + \sum_{\alpha \in A(X)} q_i(\alpha) D_\alpha. 
\end{align*}
That is, we define the entries of $L_i$ as follows:
\begin{align*}
    (L_i)_{x,y} = \frac{1}{|X|}q_i(\langle x,y\rangle), \quad i = 0, 1, \ldots, s. 
\end{align*}

The faulty proof in \cite{hoggar_tight_1984} assumes that $q_i(\alpha) = Q_i^0(\alpha)$ for $X$ any tight $(2s-\varepsilon)$-design. 
Indeed, \cite{hoggar_tight_1984} uses matrices $E_i$ instead of $L_i$:
\begin{align*}
    (E_i)_{x,y} = \frac{1}{|X|}Q_i^0(\langle x,y\rangle), \quad i = 0, 1, \ldots, s.
\end{align*}
The matrices $E_0, E_1, \ldots, E_{s-\varepsilon}$ are orthogonal idempotents. The problem, as described in \cite{lyubich_tight_2009}, is that for $\varepsilon = 1$ the matrix $E_s$ is not necessarily idempotent, so we cannot assume that $L_s = E_s$, where $s = |A(X)|$. 
To see why, note that the orthogonal basis of $s + 1$ idempotents $L_0, L_1, \ldots, L_s$ must satisfy,
\begin{align*}
    I = \sum_{i = 0}^s L_i.
\end{align*}
The components of this matrix equation are given by,
\begin{align*}
    \delta_{x,y} = \frac{1}{|X|} \sum_{i = 0}^s q_i(\langle x,y\rangle).
\end{align*}
For a tight $t$-design we also have,
\begin{align*}
    \delta_{x,y} = \frac{1}{|X|} \mathrm{ann}(\langle x,y\rangle) 
    = \frac{1}{|X|} \langle x,y\rangle^\varepsilon R_{s-\varepsilon}^\varepsilon(\langle x,y\rangle) 
    = \frac{1}{|X|} \langle x,y\rangle ^\varepsilon \sum_{i = 0}^{s-\varepsilon} Q_i^\varepsilon(\langle x,y\rangle). 
\end{align*}
This means that we require,
\begin{align*}
    \sum_{i = 0}^s q_i(\langle x,y\rangle) = \langle x,y\rangle ^\varepsilon \sum_{i = 0}^{s-\varepsilon} Q_i^\varepsilon(\langle x,y\rangle).
\end{align*}
When $\varepsilon = 0$, this constraint is satisfied by setting $q_i(\langle x,y\rangle) = Q_i^0(\langle x,y\rangle)$. When $\varepsilon = 1$ we need to select $q_s(\langle x,y\rangle)$ more carefully. 

To find $L_s$ we begin with,
\begin{align*}
    L_s = I - \sum_{i = 0}^{s-1} L_i.
\end{align*}
Since the $E_i$ are idempotent for $i\ne s$ we set $L_i = E_i$ for $i \ne s$. This yields the following components of $L_s$:
\begin{align*}
    (L_s)_{x,y} = \delta_{x,y} - \frac{1}{|X|} \sum_{i = 0}^{s-1} Q_i^0(\langle x,y\rangle).
\end{align*}
The first term is equal to $\mathrm{ann}(\langle x,y\rangle)/|X|$ and the sum in the second term is equal to $R_{s-1}^0(\langle x,y\rangle)$. This provides us with a general expression for $L_s$, regardless of whether $\varepsilon$ equals $0$ or $1$:
\begin{align*}
    (L_s)_{x,y} = \frac{1}{|X|} \left( \mathrm{ann}(\langle x,y\rangle) -  R_{s-1}^0(\langle x,y\rangle) \right).
\end{align*}
When $\varepsilon = 0$ we have $\mathrm{ann}(\langle x,y\rangle) = R^0_s(\langle x,y\rangle)$ which ensures that $L_s = E_s$. However, when for $\varepsilon = 1$ we have $L_s \ne E_s$.

\section{Idempotent Ranks and Complex Automorphisms}

Having replaced the faulty $E_i$ with a proper $L_i$ idempotent basis, as described in \cite{lyubich_tight_2009}, we return to the proof in \cite{hoggar_tight_1984}. 
Hoggar's proof involves the so-called wild automorphisms of $\mathbb{C}$.
If we admit the axiom of choice, then the automorphisms of $\mathbb{C}$ include \textit{wild automorphisms}, namely automorphisms of $\mathbb{C}$ that map $\mathbb{R}$ to a dense subset of $\mathbb{C}$ \cite{yale_automorphisms_1966}.
The only complex numbers that are fixed by all wild automorphisms are the rationals $\mathbb{Q}$. 
We use this property of $\mathbb{Q}$ below, assuming the axiom of choice. 

Let $\sigma$ be an automorphism of $\mathbb{C}$, potentially among the wild automorphisms. 
The map $\sigma$ acts as an automorphism of the the Bose-Mesner algebra by acting on all matrix coefficients. Even so, $\sigma$ leaves the basis matrices $I$ and $D_\alpha$ fixed, since they only have $0$ and $1$ for entries.
Since each $L_i$ is a $\mathbb{C}$-linear combination of the $D_\alpha$ matrices, and since the orthogonal idempotent basis is unique, 
the action of $\sigma$ on $\{L_0, L_1, \ldots, L_s\}$ must permute these idempotent matrices. It must also preserve matrix rank, so that $\mathrm{rank}~L_i = \mathrm{rank}~\sigma(L_i)$.

\begin{lemma}
    \cite{hoggar_tight_1984}
    If the idempotent matrices $\{L_0, L_1, \ldots, L_s\}$ have distinct ranks then $\langle x, y\rangle$ is rational for all $x,y$ in tight $t$-design $X$.
    \label{ranksRational}
\end{lemma}

\begin{proof}
    If the matrices $\{L_0, L_1, \ldots, L_s\}$ have distinct ranks then any field automorphism of $\mathbb{C}$ must fix these matrices, so that $\sigma(L_i) = L_i$. 
    This ensures that the $L_i$ are matrices with rational entries.
    Specifically, each $L_i$ is of the form $|X|(L_i)_{x,y} = q_i(\langle x,y\rangle)$.
    If $\sigma(L_i) = L_i$ then we also have $\sigma q_i(\langle x,y\rangle) = q_i(\langle x,y\rangle)$.
    If $q_i(\langle x,y \rangle)$ is fixed by all $\sigma$ then it is rational. 
    In the case of $i = 1$ we have,
    \begin{align*}
        (L_1)_{x,y} = Q_1^0(\langle x,y\rangle) = \left(\frac{1}{2}\rho d +1\right) \left(\rho \langle x,y\rangle - 1\right),
    \end{align*}
    This means that $\langle x,y\rangle$ is rational.    
    
\end{proof}

The next task is to compute the ranks of the $L_i$ idempotent matrices. The ranks found here are the same as those calculated in \cite{lyubich_tight_2009}, but presented in a slightly different form.

\begin{lemma}
    \label{specificRanks}
    Let $X$ be a tight $(2s-\varepsilon)$-design. Then the orthogonal idempotents of the Bose-Mesner algebra have the following ranks:
    \begin{align*}
        \mathrm{rank}~L_i = \left\lbrace 
         \begin{array}{c l}
                 Q_i^0(1), 
                 & i = 0,1,\ldots, s-1 \\
                 R_{s-\varepsilon}^\varepsilon (1) - R_{s-1}^0(1), & i = s
         \end{array}
        \right.
    \end{align*}
    Here we have,
    \begin{align*}
        Q_i^0(1) &= \left(\frac{\frac{1}{2}\rho d  + 2i  - 1}{\frac{1}{2}\rho d + i  - 1} \right) \frac{(\frac{1}{2}\rho d)_{i} (\frac{1}{2}\rho d - \frac{1}{2} d)_i}{(\frac{1}{2}d)_{i} i!}.
    \end{align*}
    For $\varepsilon = 0$ we have $R_{s}^0(1) - R_{s-1}^0(1) = Q_s^0(1)$. For $\varepsilon = 1$ we have,
    \begin{align*}
        R_{s-1}^1(1) - R_{s-1}^0(1) = \frac{s}{\frac{1}{2}\rho d + 2 s - 1} Q_s^0(1).
    \end{align*}
\end{lemma}

\begin{proof}
    The rank of an idempotent matrix is equal to its trace. 
    For $L_i$ with $i \ne s$ we have,
    \begin{align*}
        \mathrm{rank}~L_i = \mathrm{Tr}~L_i = 
        \sum_{x \in X}(L_i)_{x,x} 
        = \sum_{x \in X} \frac{Q_i^0(1)}{|X|} = Q_i^0(1).
    \end{align*}
    For $L_s$ we have,
    \begin{align*}
        \mathrm{rank}~L_s = \mathrm{Tr}~L_s = 
        \sum_{x \in X}(L_s)_{x,x} 
        = \sum_{x \in X} \frac{1}{|X|} \left( \mathrm{ann}(1) -  R_{s-1}^0(1) \right).
    \end{align*}
    Since $\mathrm{ann}(1) = 1^\varepsilon R_{s-\varepsilon}^\varepsilon (1)$ we have, 
    \begin{align*}
        \mathrm{rank}~L_s = R_{s-\varepsilon}^\varepsilon (1) - R_{s-1}^0(1).
    \end{align*}
    The specific expression for $Q_i^0(1)$ given above is obtained from the expression for $Q_k^\varepsilon(x)$ given earlier and the property $P_k^{(\alpha, \beta)} = \binom{\alpha+k}{k} = \frac{(\alpha + 1)_k}{k!}$.
    Next we use the following expression from \cite{lyubich_tight_2009}:
    \begin{align*}
         R_{s-\varepsilon}^\varepsilon(1) =  \frac{(\frac{1}{2}\rho d)_s (\frac{1}{2}\rho d - \frac{1}{2}d   + 1)_{s-\varepsilon}}{(\frac{1}{2}d)_s (s-\varepsilon)!} 
         = \frac{(\frac{1}{2}\rho d)_s}{(\frac{1}{2}d)_s} \binom{\frac{1}{2}d(\rho - 1)  + s - \varepsilon}{s - \varepsilon}
    \end{align*}
    By construction $R_{s}^0(1) - R_{s-1}^0(1) = Q_s^0(1)$.
    When $\varepsilon = 1$ we instead have,
    \begin{align*}
       R_{s-1}^1(1) - R_{s-1}^0(1) &= \left(\frac{(\frac{1}{2}\rho d)_s}{(\frac{1}{2}d)_s}   -\frac{(\frac{1}{2}\rho d)_{s-1}}{(\frac{1}{2}d)_{s-1}}\right)
       \binom{\frac{1}{2}d(\rho - 1)  + s - 1}{s - 1}.
    \end{align*}
    Using $\binom{a}{b-1} = \frac{b}{a-b+1}\binom{a}{b}$ and then $\binom{a+b}{b} = \frac{(a+1)_b}{b!}$ we have,
    \begin{align*}
       R_{s-1}^1(1) - R_{s-1}^0(1) 
        &=  \frac{(\frac{1}{2}\rho d)_s}{(\frac{1}{2}d)_s}\left(1 - \frac{\frac{1}{2}d + s - 1}{\frac{1}{2}\rho d+s-1} \right)
        \frac{s}{\frac{1}{2}d(\rho - 1)  }\binom{\frac{1}{2}d(\rho - 1)  + s - 1}{s} \\
        &= \frac{s}{\frac{1}{2}\rho d + s - 1}   \frac{(\frac{1}{2}\rho d)_s (\frac{1}{2}\rho d - \frac{1}{2} d)_s}{(\frac{1}{2}d)_s s!} \\
        &= \frac{s}{\frac{1}{2}\rho d + 2 s - 1} Q_s^0(1).
    \end{align*}
\end{proof}


The exceptional case of the unit circle, $\Omega_2 \cong \mathbb{RP}^1$, deserves specific attention. 
We examine it before proceeding with the remainder of the proof. 

\begin{lemma}
    \label{circleranks}
    When rank $\rho = 2$ and degree $d = 1$, the case of $\Omega_2 \cong \mathbb{RP}^1$, we have 
    \begin{align*}
        \mathrm{rank}~L_0 = 1,~
        \mathrm{rank}~L_i = 2, \quad i = 1, \ldots, s-\varepsilon.
    \end{align*}
    When $\varepsilon = 1$, we have $\mathrm{rank}~L_s = 1$. 
\end{lemma}

\begin{proof}
We evaluate $Q_i^0(1)$ for $\rho = 2$ and $d=1$. First, $Q_0^0(x) = 1$ so we have $\mathrm{rank}~L_0 = Q_0^0(1) = 1$. In what follows we assume $i > 0$. 
The following expression simplifies to $2$:
\begin{align*}
    Q_i^0(1) &= \left(\frac{1  + 2i  - 1}{1 + i  - 1} \right) \frac{(1)_{i} (1 - \frac{1}{2} )_i}{(\frac{1}{2})_{i} i!} = 2.
\end{align*}
When $\varepsilon = 1$ we have,
\begin{align*}
    \mathrm{rank}~L_s = R_{s-1}^1(1) - R_{s-1}^0(1) = \frac{s}{1 + 2 s - 1} Q_s^0(1)  = \frac{2}{2} = 1.
\end{align*}

\end{proof}


The remaining cases satisfy the following lemma:

\begin{lemma}
    \cite{lyubich_tight_2009}
    For any tight $(2s-\varepsilon)$-design, not in $\Omega_2 \cong \mathbb{RP}^{1}$, the idempotent basis matrices satisfy,
    \begin{align*}
        \mathrm{rank}~L_0 < \mathrm{rank}~L_1 < \cdots < \mathrm{rank}~L_{s-\varepsilon}. 
    \end{align*}
    \label{ascendinglemma}
\end{lemma}

\begin{proof}
    This is equivalent to $Q_0^0(1) < Q_1^0(1) < \cdots < Q_{s-\varepsilon}^0(1)$.
    We can compute the following expression:

    \begin{align*}
        Q_{i+1}^0(1) &= \left(\frac{\frac{1}{2}\rho d  + 2i  + 1}{\frac{1}{2}\rho d  + 2i  - 1 } \right)\left(\frac{\frac{1}{2}d (\rho - 1) + i}{\frac{1}{2}d + i }\right) \left(\frac{\frac{1}{2}\rho d + i  - 1}{i+1}\right) Q_i^0(1)
    \end{align*}
    The first factor is always greater than one. 
    The second factor is always greater than or equal to one. 
    The third factor is greater than or equal to one when $\rho d \ge 4$.  
    This leaves the $(\rho,d) = (3,1)$ case to check, which yields $Q_{i+1}^0(1) = \left(2i + \frac{5}{2}\right)/\left(2i + \frac{1}{2}\right) Q_i^0(1)$. 
    Therefore, for $(\rho,d) \ne (2,1)$ we have $Q_i^0(1) < Q_{i+1}^0(1)$ for all $i\ge 0$. 
\end{proof}


\section{Completing the Proof}

To complete the proof of Theorem \ref{maintheorem} we need to apply Lemma \ref{ranksRational} to all the relevant cases. 
However, since $\Omega_2 \cong \mathbb{RP}^1$ generally involves idempotent basis matrices of equal rank, we cannot use Lemma \ref{ranksRational} for this exceptional case.
This case is examined carefully for completeness in \cite{lyubich_tight_2009} and we address it in the following theorem.

\begin{theorem}
    \label{circlecase} \cite{lyubich_tight_2009}
    A tight $t$-design in $\Omega_2 \cong \mathbb{RP}^1$ has a rational angle set if and only if $t = 1,2,3,5$. 
\end{theorem}

\begin{proof}
    A tight $t$-design in $\Omega_2 \cong \mathbb{RP}^1$ always exists and is given by the corners of a regular $(t+1)$-gon. 
    The angle between any pair of design points on the unit circle is $\theta = 2 m \pi/(t+1)$ for some integer $m$. The corresponding Jordan inner product is $\langle x, y\rangle = \cos^2\left(\frac{\theta}{2}\right) = \frac{1}{2}+\frac{1}{2}\cos\theta$.
    This means that $\langle x,y\rangle$ is rational if and only if $\cos \theta$ is rational. The only values of $t$ for which $\cos\theta$ is rational are known to be $t = 1,2,3,5$. 
\end{proof}

We now examine the remaining cases with $(\rho,d) \ne (2,1)$, i.e., distinct from the unit circle. The simplest to deal with, using Lemma \ref{ranksRational}, is the case where $t = 2s$ is even. 

\begin{theorem}
    A tight $(2s)$-design, not in $\Omega_2 \cong \mathbb{RP}^1$, has a rational angle set.
\end{theorem}

\begin{proof}
    According to Lemma \ref{ascendinglemma}, the idempotent basis matrices $L_i$ each have distinct rank. Therefore, by Lemma \ref{ranksRational}, the angle set is rational.
\end{proof}

We now examine cases with odd $t = 2s-1$. 

\begin{lemma}
    \label{ranksLsL1}
    Let $(\rho,d) \ne (2,1)$. If $L_s$ and $L_1$ have distinct ranks then the angle set is rational. 
\end{lemma}

\begin{proof}
    According to Lemma \ref{ascendinglemma}, the ranks of all $L_i$ are distinct except possibly for $L_s$ when $\varepsilon = 1$.  
    Therefore no field automorphism of $\mathbb{C}$ interchanges $L_1$ with any $L_i$ other than possibly $L_s$. 
    If $L_s$ and $L_1$ have distinct ranks, then $L_1$ is fixed by all field automorphisms of $\mathbb{C}$ and therefore is a matrix with rational entries $Q_1^0(\alpha)/|X|$. As described in the proof of Lemma \ref{ranksRational}, it follows from the rationality of $Q_1^0(\alpha)$ that $\alpha$ is rational.
\end{proof}

The simplest odd $t = 2s-1$ case is for $s=1$, corresponding to a tight $1$-design. A tight $1$-design exists in each $\Omega_{d+1}$ and $\mathbb{FP}^{\rho-1}$ and is also known as a \textit{Jordan frame}, or full rank set of orthogonal primitive idempotents. The annihilator polynomial is $\mathrm{ann}(x) = \rho x$, so the angle set is $A = \{0\}$, which is clearly rational. 
Since this case is fully understood, we will assume $s > 1$ in what follows.

We now address the remaining spherical cases.

\begin{theorem}
    A tight $(2s-1)$-design in $\Omega_{d+1}$ with $d > 1$ has a rational angle set except when $d = 2$ and $s = 3$.
\end{theorem}

\begin{proof}
    In the spherical cases we have $\rho = 2$ and can simplify $Q_i^0(1)$ to the following:
    \begin{align*}
    Q_i^0(1) &= \left(\frac{d  + 2i  - 1}{d + i  - 1} \right) \frac{(d)_{i} }{ i!}.
    \end{align*}
    This means that,
    \begin{align*}
     \mathrm{rank}~L_1 = Q_1^0(1) &= d+1.
    \end{align*}
    Likewise,
    \begin{align*}
        \mathrm{rank}~L_s &= R_{s-1}^1(1) - R_{s-1}^0(1) = \frac{s}{ d + 2 s - 1} Q_s^0(1) 
        = \binom{d + s - 2}{s-1}.
    \end{align*}
    For $s = 2$, we always have $\mathrm{rank}~L_s - \mathrm{rank}~L_1 = -1$. By Lemma \ref{ranksRational}, a tight spherical $3$-design must therefore have rational angle set. 
    For $s > 2$, $\mathrm{rank}~L_s - \mathrm{rank}~L_1 \ge 0$ with the equality achieved only when $d = 2$ and $s=3$. Aside from this exception, Lemma \ref{ranksRational} ensures a rational angle set. 
    Therefore, the only case where $\mathrm{rank}~L_s = \mathrm{rank}~L_1$ is for $d=2$ and $s=3$. A tight $5$-design exists in this case and is known to be the vertices of a regular icosahedron in $\Omega_3$. 
\end{proof}

\begin{remark}
    In contrast to Lyubich, who only deals with projective cases of degree $d = 1,2,4$, we have dealt here with all spherical ($\rho = 2$) cases simultaneously.
    This provides slightly different approach to identifying  Lyubich's exception in $\Omega_3 \cong \mathbb{CP}^1$, namely as spherical design rather than as a projective design. 
\end{remark}

We will call a design with $\rho > 2$ a \textit{strictly projective design}. 
All remaining cases are strictly projective. 
In what follows we will therefore require that $\rho > 2$. 

\begin{theorem}
    A strictly projective tight $(2s-1)$-design, i.e., with $\rho > 2$, has a rational angle set.
\end{theorem}

\begin{proof}
    The proof method shown here is equivalent to the method used in \cite{lyubich_tight_2009}, but the parameters are allowed to extend to the octonion case ($d = 8$) yet restricted to the strictly projective $\rho > 2$ cases.
    We need to verify that $\mathrm{rank}~L_s \ne \mathrm{rank}~L_1$. 
    In the cases below we use that fact that,
    \begin{align*}
        \mathrm{rank}~L_s \ge \frac{2}{\frac{1}{2}\rho d + 3}Q_2^0(1)
    \end{align*}
    Specifically, as a function of $s$ the expression $s Q_s^0(1) /(\frac{1}{2}\rho d + 2s - 1)$ decreases as the value of $s$ decreases. 
    We can confirm this fact using the expression given in the proof of Lemma \ref{ascendinglemma}.
    We will therefore assume that $\mathrm{rank}~L_s$ is greater than or equal to the same expression evaluated at $s=2$.
    This ensures that,
    \begin{align*}
        \mathrm{rank}~L_s - \mathrm{rank}~L_1 \ge \frac{2}{\frac{1}{2}\rho d + 3}Q_2^0(1) - Q_1^0(1). 
    \end{align*}
    Writing the expression on the right hand side explicitly, in simplified form, we have,
    \begin{align*}
        \mathrm{rank}~L_s - \mathrm{rank}~L_1 \ge 
        \frac{1}{2}\left(\frac{d (\rho - 1)}{d(d+2)}\right)
        \left(\rho^2 d^2-2\rho d^2-2 d-4\right)
    \end{align*}
    Whenever the right hand side of the inequality is greater than zero, we have confirmed by Lemma \ref{ranksRational} that $\mathrm{rank}~L_s \ne \mathrm{rank}~L_1$ and therefore that the angle set is rational.
    We define, 
    \begin{align*}
        f_d(\rho) = \rho^2 d^2-2\rho d^2-2 d-4.
    \end{align*}
    If $f_d(\rho) > 0$ then the corresponding tight projective $(2s-1)$-design has a rational angle set.

    
    \subsubsection*{Real Projective Case}
    Set $d = 1$ and $\rho > 2$. Then we have,
    \begin{align*}
        f_1(\rho) = \rho^2 -2\rho -6.
    \end{align*}
    The only integer value of $\rho > 2$ with $f_1(\rho) \le 0$ is $\rho = 3$.   
    This means that the real projective case of $\rho = 3$ and $d=1$ is a possible case for $\mathrm{rank}~L_s = \mathrm{rank}~L_1$. We must examine this possibility more closely.
    
    Let $d = 1$ and $\rho =3$. If $s = 2$, for a tight $3$-design, then we know that $\mathrm{rank}~L_s \ne \mathrm{rank}~L_1$ by the fact that $f_1(3) \ne 0$, as shown above. We need to also ensure $\mathrm{rank}~L_s \ne \mathrm{rank}~L_1$ for $s>2$. To do so, we repeat the argument given above except instead we use,
    \begin{align*}
        \mathrm{rank}~L_s\ge \frac{3}{\frac{1}{2}\rho d + 5}Q_3^0(1). 
    \end{align*}
    For $d = 1$ and $\rho = 3$ we have,
    \begin{align*}
        \mathrm{rank}~L_s  - \mathrm{rank}~L_1 \ge \frac{6}{13} Q_3^0(1) - Q_1^0(1) = 1.
    \end{align*}
    Therefore all real tight strictly projective designs have rational angle set.
    
    \subsubsection*{Complex Projective Cases} Let $\rho \ge 3$ and $d = 2$. Then the second factor in the inequality above simplifies to,
    \begin{align*}
        f_2(\rho) = 4\rho^2 - 8\rho -8.
    \end{align*}
    All integer values of $\rho > 2$ satisfy $f_2(\rho) > 0$.
    This means that $\mathrm{rank}~L_s > \mathrm{rank}~L_1$ for all of the complex projective cases. 
    
    \subsubsection*{Quaternion Projective Cases} Let $\rho \ge 3$ and $d = 4$. Then the second factor in the inequality above simplifies to,
    \begin{align*}
        f_4(\rho) = 16\rho^2 - 32 \rho -12.
    \end{align*}
    All integer values of $\rho > 2$ satisfy $f_4(\rho) > 0$.
    This means that $\mathrm{rank}~L_s > \mathrm{rank}~L_1$ for all of the quaternion projective cases. 
    
    \subsubsection*{Exceptional Octonion Projective Case} Let $\rho = 3$ and $d = 8$. Then the second factor in the inequality above simplifies to,
    \begin{align*}
        f_8(3) = 172.
    \end{align*}
    This means that $\mathrm{rank}~L_s > \mathrm{rank}~L_1$ for the exceptional case.
    Having checked all the cases, we have distinct ranks for all idempotent basis matrices $L_i$ and by Lemma \ref{ranksRational} the angle sets of a strictly projective tight $(2s-1)$-design must be rational.
\end{proof}

\section{Conclusion}

We have extended the result of \cite{hoggar_tight_1984} for $d = 1,2,4,8$, which is corrected by \cite{lyubich_tight_2009} for $d = 1,2,4$, to the full range of possible values of rank $\rho$ and degree $d$, proving Theorem \ref{maintheorem}.
This clarifies the full conditions under which a tight $t$-design---whether spherical or projective---has rational angle set. The only examples of irrational angles sets exist on the unit circle $\Omega_2 \cong \mathbb{RP}^{1}$ for $t \ne 2,3,4,5$ and on the unit sphere $\Omega_3\cong \mathbb{CP}^1$ for $t = 5$. 

\bibliographystyle{amsalpha}
\bibliography{references}

\appendix

\end{document}